\documentclass[11pt,reqno]{amsart}
\usepackage[colorlinks=true,allcolors=blue,backref=page]{hyperref}
\usepackage{color}
\usepackage{booktabs}
\usepackage{longtable}
\usepackage{array}
\usepackage{amsmath, amssymb, amsthm}

\usepackage{mathrsfs}
\usepackage{mathtools}
\usepackage[noabbrev,capitalize,nameinlink]{cleveref}
\crefname{equation}{}{}
\usepackage{fullpage}
\usepackage[noadjust]{cite}
\usepackage{graphics}
\usepackage{pifont}
\usepackage{tikz}
\usepackage{tikz-cd}
\usepackage{bbm}
\usepackage[T1]{fontenc}

\usetikzlibrary{arrows.meta}

\usepackage{environ}
\usepackage{framed}
\usepackage{url}
\usepackage[linesnumbered,ruled,vlined]{algorithm2e}
\usepackage[noend]{algpseudocode}
\usepackage[labelfont=bf]{caption}
\usepackage{cite}
\usepackage{framed}
\usepackage[framemethod=tikz]{mdframed}
\usepackage{appendix}
\usepackage{graphicx}
\usepackage[textsize=tiny]{todonotes}
\usepackage{tcolorbox}
\usepackage{enumerate}
\usepackage[shortlabels]{enumitem}
\allowdisplaybreaks[1]

\apptocmd{\sloppy}{\hbadness 10000\relax}{}{} 

\crefname{algocf}{Algorithm}{Algorithms}

\crefname{equation}{}{} 
\crefname{conjecture}{Conjecture}{Conjectures} 
\AtBeginEnvironment{appendices}{\crefalias{section}{appendix}} 

\crefformat{enumi}{#2#1#3}
\crefrangeformat{enumi}{#3#1#4 to~#5#2#6}
\crefmultiformat{enumi}{#2#1#3}%
{ and~#2#1#3}{, #2#1#3}{ and~#2#1#3}

\usepackage[final,color]{showkeys} 

\colorlet{refkey}{orange!20}
\colorlet{labelkey}{blue!30}

\crefname{algocf}{Algorithm}{Algorithms}

\numberwithin{equation}{section}
\newtheorem{theorem}{Theorem}[section]

\newtheorem{lemma}[theorem]{Lemma}

\crefname{claim}{Claim}{Claims}

\newtheorem*{question*}{Question}

\theoremstyle{definition}

\newtheorem*{definition*}{Definition}

\theoremstyle{remark}

\newcommand{\abs}[1]{\left\lvert#1\right\rvert}

\newcommand{\snorm}[1]{\lVert#1\rVert}

\newcommand{\one}{\mathbbm{1}}

\newcommand{\mb}{\mathbb}

\newcommand{\mc}{\mathcal}

\newcommand{\Z}{\mathbb{Z}}

\title{Short proofs in combinatorics and number theory}
\author{
Boris Alexeev, Moe Putterman, Mehtaab Sawhney, Mark Sellke, and Gregory Valiant\\
OpenAI
}

\email{\{balexeev,mputt,msawhney,msellke,valiant\}@openai.com}
\date{}

\begin{document}
\begin{abstract}
We give a triplet of short proofs, each of which answers a question raised by Erd\H{o}s. The first concerns the small prime factors of $\binom{n}{k}$, the second concerns whether an additive basis $A$ can always be split into pieces $A_1$ and $A_2$ such that each of $A_i + A_i$ has bounded gaps, and the final concerns whether $\{\alpha p\}$ is ``well-distributed'' in the sense introduced by Hlawka and Petersen. In each case, the proof is due entirely to an internal model at OpenAI.
\end{abstract}
\maketitle
\vspace{-2em}
\section{Introduction}
We study a triplet of questions raised by Erd\H{o}s. Due to the brevity of the proofs involved, and in the spirit of a sequence of papers by Alon \cite{Alon03,Alon08,Alon16} and by Conlon, Fox, and Sudakov \cite{CFS14,CFS16}, we have chosen to present them in a single manuscript. We now briefly survey the results discussed in each section.

Section 2 studies the small prime factors of the binomial coefficient $\binom{n}{k}$. In particular, let $u(n,k) = \prod_{p\le k}p^{v_p\big(\binom{n}{k}\big)}$ and let $f(n)$ be the minimal $k$ such that $u(n,k)>n^2$. We prove that $f(n)$ is at most polylogarithmic in size and that $f(n)$ is infinitely often logarithmic in size.

Section 3 studies the question of whether a basis $A$ can always be split into $A_1\sqcup A_2$ such that $A_1+A_1$ and $A_2+A_2$ both have bounded gaps. We answer this question in the negative via an explicit construction.

Finally, in Section 4 we consider whether $\{\alpha p\}$ is ``well-distributed'' in the sense of Hlawka \cite{Hla55} and Petersen \cite{Pet56}. We call a sequence $x_1, x_2,\ldots$ ``well-distributed'' if \[\lim_{k\to \infty}\sup_{\substack{n\ge 1\\I:=[a,b]\subseteq [0,1]}} \frac{|\#\{n<m\le n+k: x_m\in I\} - (b-a)\cdot k|}{k} = 0.\]
Answering a question of Erd\H{o}s, we note that Dirichlet's approximation theorem, together with the celebrated work of Maynard \cite{May15,May16} and Tao (in particular, via a corollary of Banks, Freiberg, and Turnage-Butterbaugh \cite{BFT15}), implies that $\{\alpha p_n\}$ is not well-distributed. The key point is that one can find long strings of consecutive primes for which $\{\alpha p_n\}$ lies in a narrow interval. We also note that work of Shao and Ter\"{a}v\"{a}inen \cite{ShaoTer21} very likely provides an alternative route: they prove the slightly weaker statement that for any interval $I$ there are bounded gaps between strings of primes with $\{\alpha p\}\in I$ (see also the discussion in \cite[pg.~2]{BHHMW23} and work of Benatar \cite{Benatar13} which proves the analog of \cite{ShaoTer21} for Diophantine $\alpha$). 

\subsection*{Comment on the use of AI}
The proofs in this manuscript are due entirely to an internal model at OpenAI. The role of the human authors was simply to digest the proofs and modify the write-ups for clarity and elegance.

Additionally, after obtaining these proofs, the authors were motivated to see whether GPT 5.4 Pro could recreate them; the upper bound in the first result can be obtained \href{https://chatgpt.com/share/69c9d0d0-13ac-832b-a995-16d9f6674954}{here}, and the third result can be obtained \href{https://chatgpt.com/share/69c9d16c-0e24-832f-8cce-4fe4f6b7bb2b}{here}.
(In each case, we performed fewer than $10$ identical attempts with the Pro model; each of these problems had one other similar solution, while our second result was not solved by the Pro model.)

\section{Thresholds for binomial coefficients}
In a problem paper, Erd\H{o}s \cite{Erd79} asked about the distribution of small prime divisors of $\binom{n}{k}$. Specifically, given $0\le k\le n$, define 
\[u(n,k) = \prod_{p\le k}p^{v_p\big(\binom{n}{k}\big)}.\]
Then define 
\[f(n) = \min\{0\le k\le n: u(n,k)>n^2\}.\]
Using recent work on gaps between primes of Guth and Maynard \cite{GM24}, Tang and ChatGPT (see \cite[Problem~\#684]{BloWeb}) have shown that $f(n)\le n^{30/43+o(1)}$, and the same argument gives $f(n)\le n^{2/3+o(1)}$ under the Riemann Hypothesis (or Density Hypothesis). Via a substantially more elementary argument, we give a polylogarithmic upper bound and, for an infinite sequence of $n$, a logarithmic lower bound.

Before diving into the proof, we briefly summarize the key details behind the upper bound. By Legendre's formula, if $k\pmod p > n\pmod p$ then $p|\binom{n}{k}$. As $k$ ranges over an interval of length $p$, the number of $k$ that satisfy this condition is $(p - 1 - n\pmod p)$; the adversarial case, therefore, is that for all $p$ of a given size $Y$, $n\pmod p$ is very close to $p$. This phenomenon occurring simultaneously for most $p$ is ruled out by noting that $n \pmod p\ge p - A$ implies that $p\mid(n+1)\cdots(n+A)$.

\begin{theorem}\label{thm:binomial-main}
For $n$ sufficiently large, we have that
\[
f(n)\le \left(\frac{24}{\pi^2-6}+o(1)\right)(\log n)^2
\leq 6.20219 (\log n)^2.
\]
Furthermore, there exists a sequence $n_j\to \infty$ such that
\[
f(n_j)\ge \left(\frac12+o(1)\right)\log n_j.
\]
\end{theorem}
\begin{proof}
By Legendre's formula,
\[
v_p\bigg(\binom{n}{k}\bigg)
= \sum_{t=1}^{\infty}\bigg\lfloor \frac{n}{p^{t}}\bigg\rfloor-\bigg\lfloor \frac{n-k}{p^{t}}\bigg\rfloor-\bigg\lfloor \frac{k}{p^{t}}\bigg\rfloor
\ge \bigg\lfloor \frac{n}{p}\bigg\rfloor-\bigg\lfloor \frac{n-k}{p}\bigg\rfloor-\bigg\lfloor \frac{k}{p}\bigg\rfloor
= \one_{n\pmod p< k\pmod p}.
\]

Fix $\varepsilon>0$, and set
\[
C=\frac{24}{\pi^2-6}+\varepsilon,\qquad Y=\lfloor C(\log n)^2\rfloor.
\]
For each integer $j\ge 2$, let
\[
\mc{P}_j=\{p\text{ prime and }p\le Y/j\}.
\]
Define $r_p\equiv n\pmod p\in [0,p)$, $a_p=p-r_p$, and
\[
T_j=\sum_{p\in \mc{P}_j}\log p,\qquad
R_j=\sum_{p\in \mc{P}_j}a_p\log p,\qquad
M_j=\bigg\lfloor \frac{Y}{j\log n}\bigg\rfloor.
\]

The crucial observation is that for $1\le A\le M_j$ we have
\[
\sum_{\substack{p\in \mc{P}_j\\a_p\le A}}\log p
\le \log\Big(\prod_{m=1}^{A}(n+m)\Big)
\le A\log(n+A)
\le A(\log n + \log M_j);
\]
indeed, if $a_p\le A$ then $p$ divides one of $n+1,\ldots,n+A$. Now note that
\begin{align*}
R_j
= \sum_{A\ge 0}\Big(\sum_{p\in \mc{P}_j}\log p\cdot (1-\one_{a_p\le A})\Big)
&\ge M_jT_j-\sum_{A=0}^{M_j-1}A(\log n + \log M_j)\\
&\ge M_jT_j-\frac{M_j^2(\log n + \log M_j)}{2}\\
&= \left(\frac{C^2}{2j^2}-o(1)\right)(\log n)^3.
\end{align*}
Here we have used that $T_j=(1+o(1))Y/j$ by the prime number theorem.

Now fix $J\ge 2$. For each prime $p\le Y$, the interval $[p,Y]$ contains exactly $\lfloor Y/p\rfloor-1$ disjoint blocks of the form
\[
[mp,(m+1)p-1]\qquad (1\le m\le \lfloor Y/p\rfloor-1),
\]
and on each such block the residue classes modulo $p$ run through all of $[0,p)$ once. Hence
\begin{align*}
\sum_{k=1}^{Y}\log(u(n,k))
&\ge \sum_{k=1}^{Y}\sum_{p\le k}\log(p)\cdot \one_{n\pmod p< k\pmod p}\\
&\ge \sum_{p\le Y}\Big(\Big\lfloor \frac{Y}{p}\Big\rfloor-1\Big)(a_p-1)\log p\\
&\ge \sum_{j=2}^{J}\sum_{p\in \mc{P}_j}(a_p-1)\log p\\
&= \sum_{j=2}^{J}(R_j-T_j)\\
&\ge \left(\frac{C^2}{2}\sum_{j=2}^{J}\frac1{j^2}-o(1)\right)(\log n)^3.
\end{align*}
Therefore
\[
\frac{\sum_{k=1}^{Y}\log(u(n,k))}{Y}
\ge \left(\frac{C}{2}\sum_{j=2}^{J}\frac1{j^2}-o(1)\right)\log n.
\]
Since
\[
\sum_{j=2}^{\infty}\frac1{j^2}=\frac{\pi^2}{6}-1
\]
and
\[
C>\frac{4}{\sum_{j=2}^{\infty}j^{-2}}=\frac{24}{\pi^2-6},
\]
we may choose $J$ so that
\[
\frac{C}{2}\sum_{j=2}^{J}\frac1{j^2}>2.
\]
Thus for $n$ sufficiently large at least one $1\le k\le Y$ satisfies $u(n,k)>n^2$, proving the upper bound.


For the lower bound, for $K\ge 2$ define 
\[M_K = \prod_{p\le K}p^{\lfloor \log_p K\rfloor + 1}.\]
We prove that $f(M_K - 1)>K$. Observe that for every $0\le k\le K$ and every prime power $p^a$, we have
\[M_K - 1 \pmod {p^a} \ge k\pmod{p^a}.\]
For $a\le \lfloor \log_p K\rfloor + 1$ this is immediate, since $M_K - 1 \equiv p^{a}-1 \pmod {p^a}$. For larger $a$, we have $M_K - 1 \pmod {p^a}\ge p^{\lfloor \log_p K\rfloor + 1} - 1>K\ge k \pmod{p^a}$, as desired.

Via \[v_p\bigg(\binom{n}{k}\bigg) = \sum_{t=1}^{\infty}\bigg\lfloor \frac{n}{p^{t}}\bigg\rfloor-\bigg\lfloor \frac{n-k}{p^{t}}\bigg\rfloor-\bigg\lfloor \frac{k}{p^{t}}\bigg\rfloor = \sum_{t=1}^{\infty}\one_{n\pmod{p^t}<k\pmod{p^t}},\]
we have that $v_p(\binom{M_K-1}{k}) = 0$ for $p\le K$ and $k\le K$. Noting that $\log(M_K) = \sum_{p\le K}(\lfloor \log_p K\rfloor + 1)\log p = 2K + o(K)$ via the prime number theorem, we immediately obtain the desired result.
\end{proof}

\section{Basis of order two with no syndetic splits}

Burr and Erd\H{o}s (see Erd\H{o}s \cite{Erd94} and \cite[Problem~\#741]{BloWeb}) asked whether there exists a basis $A$ of order $2$ such that for every partition $A=A_1\sqcup A_2$, the two sumsets $A_1+A_1$ and $A_2+A_2$ cannot both have bounded gaps (equivalently, cannot both be syndetic). Erd\H{o}s wrote that he believed he had such a construction, but could not complete the proof. We provide an explicit construction here.

\begin{theorem}\label{thm:basis-main}
There exists a set $A\subset \mb{N}$ such that $A$ is a basis of order $2$, and for every partition $A=A_1\sqcup A_2$, at least one of $A_1+A_1$ and $A_2+A_2$ does not have bounded gaps.
\end{theorem}

The proof is via an explicit construction. We adopt the shorthand that for integers $x\le y$, we write $[x,y]=\{n\in\Z:x\le n\le y\}$. Define
\[
A=[2,3]\cup \bigcup_{k\ge 1}\bigl(\{c_k\}\cup B_k\cup F_k\bigr),
\]
where
\[
c_k=4\cdot 5^{k-1},
\qquad
B_k=[5\cdot 5^{k-1},\,6\cdot 5^{k-1}-1],
\qquad
F_k=[10\cdot 5^{k-1}-1,\,15\cdot 5^{k-1}].
\]
For $k\ge 0$, set
\[
A_k=[2,3]\cup \bigcup_{i=1}^k\bigl(\{c_i\}\cup B_i\cup F_i\bigr).
\]
While this is not the same as the construction of Erd\H{o}s and Nathanson in their 1975 paper \cite{ErdNat75}, it bears a great deal of resemblance.

We first note that $A$ is a basis of order $2$.
\begin{lemma}\label{lem:basis-cover}
For every $k\ge 0$,
\[[4,\,6\cdot 5^k]\subset A_k+A_k.
\]
\end{lemma}

\begin{proof}
We prove the result via induction on $k$. For $k=0$, note that $[2,3]+[2,3]=[4,6]$.

Assume $k\ge 1$ and the claim holds for $k-1$. Put $Q=5^{k-1}$. First note that
\[
I:=[2Q,3Q]_{\Z}\subset A_k,
\]
since $I=[2,3]$ for $k=1$ and $I\subset F_{k-1}$ for $k\ge 2$.
Thus the following intervals are all contained in $A_k+A_k$:
\begin{align*}
I+I&=[4Q,6Q],&
I+\{c_k\}&=[6Q,7Q],&
I+B_k&=[7Q,9Q-1],\\
\{c_k\}+B_k&=[9Q,10Q-1],&
B_k+B_k&=[10Q,12Q-2],&
I+F_k&=[12Q-1,18Q],\\
B_k+F_k&=[15Q-1,21Q-1],&
F_k+F_k&=[20Q-2,30Q].
\end{align*}
These intervals overlap consecutively, so their union contains $[4Q,30Q]$. By the inductive hypothesis, $[4,6Q]\subseteq A_k+A_k$, and since $4Q\le 6Q$, we get
\[
[4,30Q]\subseteq A_k+A_k. \qedhere
\]
\end{proof}

\cref{lem:basis-cover} implies that $A$ is an additive basis of order $2$. The crucial feature is to show that elements of the form $[9 \cdot 5^{k-1},10 \cdot 5^{k-1}-1]$ arise in a rigid manner.

\begin{lemma}\label{lem:rigid}
For each $k\ge 1$, define
\[
J_k=[9\cdot 5^{k-1},\,10\cdot 5^{k-1}-1].
\]
If $n\in J_k$ and $n=a+b$ with $a,b\in A$, then one of $a,b$ is $c_k$ and the other lies in $B_k$.
\end{lemma}

\begin{proof}
Fix $k\ge 1$ and put $Q=5^{k-1}$. Let $A_{<k}$ be the union of $[2,3]$ and the stages $1,\dots,k-1$. The largest element of $A_{<k}$ is $15\cdot 5^{k-2}=3Q$, while every element from a stage later than $k$ is at least $c_{k+1}=20Q$.

If both summands lie in $A_{<k}$, then their sum is at most $6Q$. If one summand lies in $A_{<k}$ and the other lies in $\{c_k\}$, $B_k$, or $F_k$, then the sum lies respectively below $9Q$, below $9Q$, or above $10Q-1$. Thus no representation of an element of $J_k$ uses a point of $A_{<k}$.

Among sums from stage $k$ itself, the only ones meeting $J_k=[9Q,10Q-1]$ are
\[
c_k+B_k=[9Q,10Q-1].
\]
Any sum involving a later stage is at least $20Q+2$. This exhausts all possibilities and proves the lemma.
\end{proof}

\begin{proof}[Proof of \cref{thm:basis-main}]
Now let $A=A_1\sqcup A_2$ be any partition. If $c_k\in A_1$, then \cref{lem:rigid} implies
\[
J_k\cap (A_2+A_2)=\varnothing,
\]
and if $c_k\in A_2$, then
\[
J_k\cap (A_1+A_1)=\varnothing.
\]
One of the two colors contains infinitely many of the points $c_k$. Since $\abs{J_k}=5^{k-1}\to\infty$, the opposite monochromatic sumset has arbitrarily long gaps as desired.
\end{proof}

\section{Fractional parts of $\{\alpha p\}$}
A sequence $x_1,x_2,\ldots$ is called well-distributed if
\[\lim_{k\to \infty}\sup_{\substack{n\ge 1\\I:=[a,b]\subseteq [0,1]}} \frac{|\#\{n<m\le n+k: x_m\in I\} - (b-a)\cdot k|}{k} = 0.\]

This notion was introduced by Hlawka \cite{Hla55} and Petersen \cite{Pet56}. Erd\H{o}s \cite{Erd64} proved that if $(n_k)$ is lacunary, then $(\{\alpha n_k\})$ fails to be well-distributed for almost all $\alpha$ and also claimed a proof that $(\{\alpha p_n\})$ fails to be well-distributed for at least one irrational $\alpha$; he later retracted this claim in \cite{Erdos85}. He additionally asked in \cite{Erd64,Erdos85} whether $(\{\alpha p_n\})$ is never well-distributed for any irrational $\alpha$ (see also \cite[Problem~\#997]{BloWeb}). More recently, Champagne, L\^e, Liu, and Wooley \cite{ChampLeLiuWoo24} proved that such an irrational $\alpha$ exists.

We prove the conjecture of \cite{Erd64}.
\begin{theorem}\label{thm:997}
For every real number $\alpha$, the sequence $(\{\alpha p_n\})_{n\ge 1}$ is not well-distributed.
\end{theorem}

The main input in the proof is given by the celebrated work of Maynard \cite{May15, May16} and Tao (unpublished); the following is a corollary of \cite[Corollary~3]{BFT15}.
\begin{theorem}\label{thm:input}
Fix $m\ge 1$ and $(a,q)=1$. There exists $C_m\ge 1$ such that for infinitely many $r$,
\[
p_{r+1} \equiv p_{r+2} \equiv \cdots \equiv p_{r+m} \equiv a \pmod q
\]
and
\[
p_{r+m} - p_{r+1}\le qC_m.
\]
\end{theorem}
\begin{proof}[Proof of \cref{thm:997}]
Fix $\delta\in (0,1/2)$ and $m\ge 1$. By Dirichlet's approximation theorem, for all $Q\ge 1$ there exists $q\le Q$ and $a\in \mb{Z}$ such that $(a,q) = 1$ and
\[|\alpha - a/q|\le 1/(qQ).\]

Set $Q = \lceil \delta^{-1}C_m\rceil$. Consider $r$ such that $p_{r+1} \equiv p_{r+2} \equiv \cdots \equiv p_{r+m} \equiv a\pmod q$ with $p_{r+m} - p_{r+1}\le qC_m$. Observe that for $1\le j,j'\le m$, we have
\[\snorm{\alpha(p_{r+j} - p_{r+j'})}_{\mb{R}/\mb{Z}} = \snorm{(a/q + (\alpha - a/q))(p_{r+j} - p_{r+j'})}_{\mb{R}/\mb{Z}}\le \frac{qC_{m}}{qQ}\le \delta.\]

This gives $m$ consecutive primes $p_{r+1},\ldots,p_{r+m}$ such that $\{\alpha p_{r+j}\}_{1\le j\le m}$ lies in an interval of width $\delta$ for all $1\le j\le m$. This immediately gives the desired result. 
\end{proof}

\bibliographystyle{amsplain0}
\bibliography{main.bib}

\providecommand{\bysame}{\leavevmode\hbox to3em{\hrulefill}\thinspace}
\providecommand{\MR}{\relax\ifhmode\unskip\space\fi MR }
\providecommand{\MRhref}[2]{%
  \href{http://www.ams.org/mathscinet-getitem?mr=#1}{#2}
}
\providecommand{\href}[2]{#2}
\begin{thebibliography}{10}

\bibitem{Alon03}
Noga Alon, \emph{Problems and results in extremal combinatorics---{I}},
  Discrete Mathematics \textbf{273} (2003), 31--53.

\bibitem{Alon08}
Noga Alon, \emph{Problems and results in extremal combinatorics---{II}},
  Discrete Mathematics \textbf{308} (2008), 4460--4472.

\bibitem{Alon16}
Noga Alon, \emph{Problems and results in extremal combinatorics---{III}},
  Journal of Combinatorics \textbf{7} (2016), 319--337.

\bibitem{BHHMW23}
Sai~Sanjeev Balakrishnan, F\'elix Houde, Vahagn Hovhannisyan, Maryna Manskova,
  and Yiqing Wang, \emph{Arbitrarily long strings of consecutive primes in
  special sets}, arXiv 2311.18701.

\bibitem{BFT15}
William~D. Banks, Tristan Freiberg, and Caroline~L. Turnage-Butterbaugh,
  \emph{Consecutive primes in tuples}, Acta Arithmetica \textbf{167} (2015),
  261--266.

\bibitem{Benatar13}
Jacques Benatar, \emph{The existence of small prime gaps in subsets of the
  integers}, International Journal of Number Theory \textbf{11} (2015),
  801--833.

\bibitem{BloWeb}
T.~F. Bloom, \url{https://www.erdosproblems.com}.

\bibitem{ChampLeLiuWoo24}
J.~Champagne, T.~Le, Y.-R. Liu, and T.~D. Wooley, \emph{Well-distribution
  modulo one and the primes}, arXiv 2406.19491.

\bibitem{CFS14}
David Conlon, Jacob Fox, and Benny Sudakov, \emph{Short proofs of some extremal
  results}, Combinatorics, Probability and Computing \textbf{23} (2014), 8--28.

\bibitem{CFS16}
David Conlon, Jacob Fox, and Benny Sudakov, \emph{Short proofs of some extremal
  results {II}}, Journal of Combinatorial Theory, Series B \textbf{116} (2016),
  173--196.

\bibitem{Erd79}
Paul Erd\H{o}s, \emph{Some unconventional problems in number theory}, Acta
  Mathematica Academiae Scientiarum Hungaricae \textbf{33} (1979), 71--80.

\bibitem{Erd94}
Paul Erd\H{o}s, \emph{Some problems in number theory, combinatorics and
  combinatorial geometry}, Mathematica Pannonica \textbf{5} (1994), 261--269.

\bibitem{ErdNat75}
Paul Erd\H{o}s and Melvyn~B. Nathanson, \emph{Oscillations of bases for the
  natural numbers}, Proceedings of the American Mathematical Society
  \textbf{53} (1975), 253--258.

\bibitem{Erd64}
Paul Erd{\H{o}}s, \emph{Problems and results on diophantine approximations},
  Compositio Mathematica \textbf{16} (1964), 52--65.

\bibitem{Erdos85}
Paul Erd{\H{o}}s, \emph{Some problems and results in number theory}, Number
  Theory and Combinatorics, World Scientific, Singapore, 1985, Proceedings of
  the Japan 1984 Conference (Tokyo, Okayama, Kyoto, 1984), pp.~65--87.

\bibitem{GM24}
Larry Guth and James Maynard, \emph{New large value estimates for {D}irichlet
  polynomials}, Annals of Mathematics \textbf{199} (2024), 465--522.

\bibitem{Hla55}
Edmund Hlawka, \emph{Zur formalen theorie der gleichverteilung in kompakten
  gruppen}, Rendiconti del Circolo Matematico di Palermo \textbf{4} (1955),
  33--47.

\bibitem{May15}
James Maynard, \emph{Small gaps between primes}, Annals of Mathematics
  \textbf{181} (2015), 383--413.

\bibitem{May16}
James Maynard, \emph{Dense clusters of primes in subsets}, Compositio
  Mathematica \textbf{152} (2016), 1517--1554.

\bibitem{Pet56}
G.~M. Petersen, \emph{Almost convergence and uniformly distributed sequences},
  Quarterly Journal of Mathematics \textbf{7} (1956), 188--191.

\bibitem{ShaoTer21}
Xuancheng Shao and Joni Ter\"{a}v\"{a}inen, \emph{The {B}ombieri--{V}inogradov
  theorem for nilsequences}, Discrete Analysis \textbf{2021} (2021), Paper No.
  21, 55.

\end{thebibliography}

\end{document}